\documentclass{amsart}

\usepackage[utf8]{inputenc}
\usepackage{amsmath, amssymb}
\usepackage{hyperref}
\usepackage{cleveref}
\usepackage{url}
\usepackage{thm-restate}
\usepackage{xcolor}
\usepackage[top=1in,left=1.5in,right=1.5in,bottom=1in]{geometry}    

\usepackage{tikz-cd}
\usepackage{MnSymbol}   

\newtheorem{theorem}{Theorem}[section]
\newtheorem{lemma}[theorem]{Lemma}
\newtheorem{coro}[theorem]{Corollary}
\newtheorem{proposition}[theorem]{Proposition}
\newtheorem{claim}{Claim}

\theoremstyle{definition}
\newtheorem{defn}[theorem]{Definition}
\newtheorem{rem}[theorem]{Remark}

\newtheorem{question}[theorem]{Question}

\newcommand{\F}{\mathbb{F}}
\newcommand{\N}{\mathbb{N}}
\newcommand{\Q}{\mathbb{Q}}

\newcommand{\Z}{\mathbb{Z}}

\newcommand{\ccB}{\mathcal{B}}

\newcommand{\ccO}{\mathcal{O}}
\renewcommand{\P}{\mathcal{P}}

\newcommand{\Aut}{\mathrm{Aut}}
\newcommand{\Cay}{\mathrm{Cay}}

\newcommand{\GL}{\mathrm{GL}}
\newcommand{\Inn}{\mathrm{Inn}}
\newcommand{\Isom}{\mathrm{Isom}}
\newcommand{\MCG}{\mathrm{MCG}}
\newcommand{\Out}{\mathrm{Out}}

\DeclareMathOperator{\vcd}{\mathrm{vcd}}

\newcommand{\bdry}{\partial}

\newcommand{\Fn}{\mathbb{F}_n}
\newcommand{\growth}{\lambda}

\newcommand{\lawFree}{\mathcal{L}^{\mathrm{Free}}}
\newcommand{\lawUEG}{\mathcal{L}^{\mathrm{UEG}}}
\newcommand{\lawEither}{\mathcal{L}^*}
\newcommand{\normal}{\triangleleft}

\newcommand{\vphi}{\varphi}

\renewcommand{\bar}{\overline}

\newcommand{\abrackets}[1]{\left\langle #1 \right\rangle}

\title{Extensions of hyperbolic groups have locally uniform exponential growth}
\author{Robert Kropholler}
\author{Rylee Alanza Lyman}
\author{Thomas A. Ng}
\date{}

\begin{document}
	
	\maketitle
	
	%
	%
	
	\begin{abstract}
		We introduce a quantitative characterization of subgroup 
		alternatives modeled on the Tits alternative in terms of group laws and investigate when this property is preserved under extensions.  
		We develop a framework that lets us expand the classes of groups known to have locally uniform exponential growth to include extensions of either word hyperbolic or right-angled Artin groups
		by groups with locally uniform exponential growth.  
		From this, we deduce that the automorphism group of a torsion-free one-ended hyperbolic group has locally uniform exponential growth.
		Our methods also demonstrate that automorphism groups of torsion-free one-ended toral relatively hyperbolic groups and certain right-angled Artin groups satisfy our quantitative subgroup alternative.
	\end{abstract}

	%
	%

	\section{Introduction}


	\subsection{Quantitative subgroup alternatives}
	
	In this paper we study a variation of the Tits alternative,
	distinguishing between subgroups containing a uniformly short free basis,
	and those that satisfy a law.
	Recall that for a group $G$ with finite generating set $S$ and $N \in \N$,
	a (free) subgroup $H \le G$ is \emph{$N$-short (with respect to $S$)} if there exists words of $S$-length
	at most $N$ that generate $H$, 
	and that $G$ contains a \emph{uniformly $N$-short} free subgroup if there exists
	an $N$-short free subgroup with respect to
	every finite generating set of $G$ (see also \cite[Definition~1.1]{GJN}).  This definition motivates us to introduce the following property.
	
	\begin{defn}[Quantitative law alternatives]
		\label{quant_alternatives}
		A group $L$ satisfies the \emph{quantitative free subgroup--law alternative} (resp. \emph{quantitative uniform exponential growth--law alternative}
		when there exists a constant $N = N(L) \in \Z_{> 0}$ 
		and a group law $w_L \in \Fn$ that depends only on the group $L$
		such that any finitely-generated subgroup $H \le L$ either contains a uniformly $N$-short free subgroup (resp. has entropy (see definition (\ref{def: uniform exponential growth}) below) $\growth(H) \geq \frac{\log(3)}{N}$),
		or $H$ satisfies the law $w_L$.
	\end{defn}
	
	The quantitative free subgroup--law alternative quantifies the classical Tits alternative in two ways: first by bounding the word length of the shortest free basis elements that witnesses a non-abelian free subgroup,
	and second by noting that virtually solvable groups of bounded derived length satisfy a law involving nested commutators. 
	We denote by $\lawFree$ the class of groups satisfying the quantitative free subgroup--law alternative and by $\lawUEG$ the class of groups satisfying the quantitative uniform exponential growth alternative.  It is well-known that $\lawFree \subsetneq \lawUEG$. Indeed, free Burnside groups of sufficiently large odd exponent \cite[Theorem~2.7]{Osin:Burnside_UEG} and solvable Baumslag--Solitar groups (see \cite[Lemma~6.3]{GJN} and references therein) lie in $\lawUEG$ but not $\lawFree$.
	Several of our statements apply to either of the above classes in which case we use the notation $\mathcal{L}^*$.  
	It will be convenient to consider $\mathcal{L}^*_N \subset \mathcal{L}^*$; this is the class of subgroups satisfying either quantitative law alternative with fixed constant $N$.
	Our main result expands the class of groups known to lie in $\lawFree$. 
	Combining \Cref{relhyp_Aut:alternative} and \Cref{AutRAAGalternative} below we obtain the following theorem. 
	
	\begin{theorem}
		\label{Aut_Combo_Result}
		Let $K$ be a torsion-free one-ended group that is either hyperbolic relative to free abelian subgroups or a right-angled Artin group whose defining graph contains no separating intersection of links.  
		The automorphism group $\Aut(K)$ satifies the quantitative free-subgroup--law alternative.
	\end{theorem}
	
	\Cref{Aut_Combo_Result} effectivizes work of Charney and Vogtmann \cite[Theorem~5.5]{CharneyVogtmann:AutRAAG_ResFin+TitsAlternative} (homogeneous case) and of Horbez, \cite[Theorem~0.4]{Horbez:AutFreeProduct_TitsAlternative} who prove the Tits alternative for outer automorphism groups of right-angled Artin groups, as well as work of Sela \cite[Theorem~1.9]{Sela}, Levitt \cite[Theorem~1.1]{Levitt} and Guirardel and Levitt \cite[Theorem~1.4]{GuirardelLevitt} whose work implies a non-quantitative law alternative for automorphism groups of one-ended hyperbolic groups and one-ended toral relatively hyperbolic groups. 
	We note that the one-ended hypothesis in \Cref{Aut_Combo_Result} is crucial.  In particular, the following question is still open. 
	
	\begin{question}
		\label{OutFn_law_alternative?}
		Does $\Out(\Fn)$ satisfy a quantitative free subgroup--law alternative?
	\end{question}
	
	Partial progress due to Bering \cite[Theorem~1.2]{Bering} answers \Cref{OutFn_law_alternative?} in the affirmative for linearly growing subgroups.  
	Characterizing which groups satisfy a quantitative law alternative is closely related to understanding which finitely generated groups have uniform exponential growth.
	As a consequence of \Cref{prop:hyperbolicExtension} below, a positive answer to \Cref{OutFn_law_alternative?} would immediately imply that both $\Out(\Fn)$ and also $\Aut(\Fn)$ have locally uniform exponential growth, which we discuss next.


	\subsection{Locally Uniform Exponential Growth}

	The tools we develop to prove \Cref{Aut_Combo_Result} have several applications that expands our understanding of how groups and their subgroups grow.
	A finitely generated group $G$ with generating set $S$ has \emph{exponential growth} if 
	\[ 
	\lim\limits_{n\to\infty} \frac{\ln(B_n^S)}{n}=: \growth_{S}(G) > 0. 
	\]
	Here $B^S_n$ is the size of the ball of radius $n$ about the identity with 
	respect to the word metric associated to the generating set $S$.
	The quantity $\growth_{S}(G)$ is called the \emph{exponential growth rate} of $S$.   
	Let $\mathcal{S}$ denote the collection of finite generating sets for $G$. 
	The group $G$ is said to have {\em uniform exponential growth} if  
	\begin{equation}
	\label{def: uniform exponential growth}
	\inf\limits_{S\in \mathcal{S}}\growth_{S}(G) =: \growth(G)>0.
	\end{equation}
	The quantity $\growth(G)$ is called the \emph{entropy} of $G$.  
	Uniformly $N$-short free subgroups are a certificate for uniform exponential growth.  
	To more precisely understand the subgroup structure of extensions,
	we are interested in quantifying the entropy of a group 
	as well as the entropy of its finitely generated subgroups. 
	The following condition is sometimes called 
	\emph{uniform uniform exponential growth} \cite{Gelander, Mangahas, BreuillardFujiwara}.  
	
	\begin{defn}
		We say that $G$ has \emph{locally uniform exponential growth} 
		if there is a constant $c > 0$
		such that any finitely generated exponentially growing subgroup $H$ of $G$ has entropy $\growth(H)\geq c$.
	\end{defn}
	
	Showing that a group satisfies a quantitative subgroup alternative is one avenue towards showing locally uniform exponential growth.  
	Indeed, it leaves only understanding uniform exponential growth for subgroups that satisfy a particular law.
	An immediate consequence of \Cref{quant_alternatives} is that any group in $\lawEither_N$ for which subgroups that satisfy a law are all virtually nilpotent has locally uniform exponential growth.  
	
	Known examples of groups with locally uniform exponential growth include hyperbolic groups \cite{Koubi, ArzhantsevaLysenok} relatively hyperbolic groups with locally exponentially growing peripherals \cite{Xie} (see also \Cref{relhyp:alternative_LUEG}), groups acting without global fixed points on CAT(0) square complexes \cite{GJN, KarSageev}, 
	the mapping class group of finite type surfaces \cite{Mangahas}, and linear groups \cite{EMO, Breuillard}.
	Each of the above proofs of locally uniform exponential growth proceed by either 
	exhibiting a short free subgroups or free sub-semigroups or arguing that the chosen subgroup is virtually nilpotent.
	We note that non-virtually nilpotent solvable groups are known to have uniform exponential growth \cite{Alperin, Osin:solvable}, but the authors are unaware if these results can be extended to show locally uniform exponential growth.
	The following is a key result needed to prove \Cref{Aut_Combo_Result}. 
	It says that the extension of a hyperbolic group by a group in $\lawEither$ remains in $\lawEither$.
	
	\begin{restatable*}{theorem}{hyperbolicExtension}
		\label{prop:hyperbolicExtension}
		Let $L\in \lawEither_N$ be a group with the associated law $w_L$. 
		Let $H$ be a word-hyperbolic group.
		Let $E$ be a group fitting into a short exact sequence 
		$1\to H\to E \stackrel{\vphi}\to L\to 1$. 
		Then $E$ also satisfies the same quantitative law alternative (possibly with respect to a larger constant and longer law). 
		
		Moreover, if $L$ has locally uniform exponential growth so does $E$.
		The entropy $\growth(G)$ of any finitely generated subgroup $G \leq E$ can be bounded below by a function of $N$, $|w_L|$, 
		the size of a minimal generating set of $H$, and the hyperbolicity constant $\delta$ associated to such a generating set. 
	\end{restatable*}
	
	A group is called \emph{2-free}, if every 2-generated subgroup is free. 
	Examples of 2-free groups include 
	free groups, 
	hyperbolic surface groups, 
	fundamental groups of hyperbolic 3-manifolds with $\operatorname{rank}(H_1(M;\Q)) \geq 3$ \cite[Corollary~1.9]{ShalenWagreich}, 
	and hyperbolic residually free groups \cite{Baumslag:ResiduallyFree}.
	In the special case of hyperbolic 2-free groups we obtain the following striking result. 
	
	\begin{restatable*}{theorem}{twofree}
		\label{thm:twofree}
		Let $F$ be a 2-free word-hyperbolic group and $A$ be an abelian group.
		Suppose that $E$ is a group fitting into a short exact sequence  $1 \to F \to E \to A \to 1$.
		Let $T \subset E$ be a finite generating set for an exponentially growing subgroup. 
		Then there are words of $T$-length at most $6$ that generate a free group. 
	\end{restatable*}
	
	The following corollary was obtained simultaneously and independently by Bregman and Clay \cite[Lemma~5.2]{BregmanClay} in the free-by-cyclic case.
	
	\begin{restatable*}{coro}{freebycyclic}
		\label{cor:free-by-cyclic}
		Let $G$ be a free-by-cyclic or a (hyperbolic surface)-by-cyclic group. 
		Let $T$ be a finite generating set for an exponentially growing subgroup. 
		Then there are words of length $\leq 6$ in $T$ that generate a free subgroup. 
	\end{restatable*}
	
	Important examples of free-by-cyclic and surface-by-cyclic groups 
	include fundamental groups of fibered hyperbolic 3-manifolds. 
	While previous work of Koubi \cite{Koubi} (in the word hyperbolic case) 
	and of Dey, Kapovich, and Liu \cite{DeyKapovichLiu} also produce short free subgroups,
	their bounds on word length are larger 
	and less explicit with dependence on the hyperbolicity constant of the group 
	or the Margulis number of the manifold, neither of which are easily computable from a given group presentation.  
	The proof of \Cref{cor:free-by-cyclic} 
	only relies on free and surface groups being \emph{2-free} and that cyclic groups are abelian.

	The proofs of \Cref{Aut_Combo_Result}, \Cref{prop:hyperbolicExtension}, \Cref{thm:twofree} rely on a common framework.  We work with extensions $E$ fitting into a short exact sequence
	\begin{center}
		\begin{tikzcd}
			1 \arrow[r] & K \arrow[r, "\iota"] & E \arrow[r, "\varphi"] & Q \arrow[r] &   1
		\end{tikzcd}
	\end{center}
	where the quotient, $Q$ satisfies a quantitative law alternative.  
	Exponential growth and free subgroups pull back to the extension group, so it suffices to study subgroups $G$ of $E$ where the image of $G$ in $Q$ satisfies a law.  
	This law applied to a finite generating set $T$ for $G$ gives a finite collection of elements $W \subseteq G \cap \iota(K)$. 
	Conjugation by $T$ gives an action by automorphisms on the kernel $K$.  
	Iterating these automorphisms, we obtain an ascending chain of subgroups of $K$
	\[
	H_0 < H_1 < \cdots < H_j < \cdots
	\]
	where $H_0 = \abrackets{W}$ and $H_i = \abrackets{H_{i - 1}, T(H_{i - 1})}$.
	We proceed by characterizing how non-positive curvature of the group $K$ can be exploited to give an effective bound $B$ such that either $H_B$ has exponential growth bounded below with growth rate depending on $K$ in which case so does $G$. 
	Otherwise, the ascending chain terminates with a subgroup $H_B \normal G$ that satisfies a law, whence $G$ too satisfies a law.  
	
	The remainder of this article is organized as follows.  
	In \Cref{section:lemmas} we review some basic lemmas to understand certain properties that a group may inherit from its quotients.
	In \Cref{section:hyperbolic_extensions} we consider extensions where the kernel is word hyperbolic.  We prove \Cref{thm:twofree} both as a warm-up before proving \Cref{prop:hyperbolicExtension} and also to emphasize the role that being torsion free plays in simplifying our arguments and constants.  
	In \Cref{section:automorphism_groups} we move on to study automorphism groups.  
	We explain why it suffices to only work with finite index subgroups of the outer automorphism group.  These groups frequently contain copies of $\Z^n$ for $n \geq 2$.  To address this, we prove \Cref{thm:strongestTits}, an analog of \Cref{prop:hyperbolicExtension} for extensions of right-angled Artin groups.

	\section{Closure properties of law alternatives}
	\label{section:lemmas}

	In this section we review some elementary facts about inheritance of laws, free subgroups, and exponential growth under quotients.  
	These basic results provide an initial stepping stone for the methods we develop later to characterize when extensions satisfy a quantitative law alternative.  
	Recall that a (group) law is a word $w$ in a free group $\F_r$ that can be evaluated as a function $(\oplus_{i = 1}^r G) \to G$ for any group $G$.  A groups is said to \emph{satisfy a law $w$} when the image of the law contains only the trivial element of $G$.  There are many ways to combine group laws.  One that will be particularly useful to us is composition as functions.  
	
	\begin{defn}[Composition of laws]
		Suppose that $w_K \in \F_n \cong \F(x_1, \dotsc, x_n)$ and $w_Q \in \F_m \cong \F(y_1, \dotsc, y_m)$.
		Define the \emph{composition} of the two laws 
		\[
		w_K \circ w_Q \in \F_{nm} \cong \F(z_1^1, \dotsc, z_m^1, z_1^2 \dotsc, z_m^n)
		\] 
		to be given by  
		\[
		w_K \circ w_Q(z_1^1, \dotsc z_m^n) = w_K(w_Q(z_1^1, \dotsc, z_m^1), \dotsc, w_Q(z_1^n, \dotsc, z_m^n)).
		\]
		
	\end{defn}
	
	The setting that we will use composition will be to combine laws that are satisfied by the kernel and quotient groups of a short exact sequence. 
	In this setting, composition exhibits a law that is satisfied by the extension group.
	
	\begin{lemma}
		\label{law-by-law}
		Let $K$ and $Q$ be groups satisfying laws $w_K$ and $w_Q$ respectively.
		If $E$ is a group that fits into a short exact sequence of the form
		\begin{center}
			\begin{tikzcd}
				1 \arrow[r]
				&   K \arrow[r]
				&   E \arrow[r]
				&   Q \arrow[r]
				&   1
			\end{tikzcd}    
		\end{center}
		then $E$ satisfies the composite law $w_K \circ w_Q$.
	\end{lemma}
	
	\begin{proof}
		Each subword $w_Q(z_1^k \dotsc, z_m^k)$ has trivial image in $Q$, so lies in the image of $K$. Applying $w_K$ to these subwords gives a trivial word in $E$.
	\end{proof}
	
	An immediate consequence of \Cref{law-by-law} is that to show that a group satisfies a law it suffices to work with finite index subgroups.  This is a special case of groups that surject a finite group whose kernel satisfies a law. 
	
	\begin{coro}
		\label{virtual_law}
		Let $G$ be a group.
		Suppose that $\vphi: G \twoheadrightarrow \mathfrak{F}$ is a surjection to a finite group $\mathfrak{F}$. 
		If the kernel satisfies a law $w_{ker}$, then $G$ satisfies a law.  
		
		In particular, if $H$ is a group that satisfies a law $w_H$ and $H \le G$ with index $[G: H] \leq n$, then $G$ satisfies the law $w_H \circ x^{n!}$.   
	\end{coro}
	\begin{proof}
		The group $G$ naturally fits into a short exact sequence $1 \to \ker(\vphi) \to G \to \mathfrak{F} \to 1$.  Every finite group satisfies some law by Lagrange's theorem, so by \Cref{law-by-law} $G$ satisfies a law.  
		
		When $H$ is a finite index subgroup of $G$, The group $G$ acts on the set of cosets of $H$.  This gives a homomorphism $G \to S_n$ to the symmetric group on $n$ letters.  The kernel is a subgroup of $H$, so also satisfies $w_H$.  Likewise, the image in $S_n$ will satisfy the law $x^{n!}$. 
		We apply \Cref{law-by-law} to complete the proof. 
	\end{proof}
	
	Laws are not the only thing that are inherited from quotients.  From the other side of the dichotomy in the quantitative law alternatives, we can try to understand free groups and growth from quotients.
	
	\begin{lemma}\label{shortfreeinquotient}
		Suppose that $\varphi\colon G\to H$ is a homomorphism. 
		Let $S$ be a subset of $G$. 
		If $\varphi(S) = \{\varphi(s)\mid s\in S\}$ generates an $N$-short free subgroup, then $S$ generates an $N$-short free subgroup. 
	\end{lemma}
	\begin{proof}
		We can assume the free subgroup $K<H$ has rank 2. 
		Let $h_1, h_2$ be elements of $\langle \varphi(S)\rangle$ of length $\leq N$ that generate $K$. 
		Then we can find words $g_1, g_2$ in $\langle S\rangle$ of length $\leq N$ that map onto $h_1$ and $h_2$. 
		By the universal property of free groups we see that $g_1, g_2$ must generate a free group of rank 2 which is $N$-short. 
	\end{proof}
	
	At times we will be interested in groups that have uniform exponential growth but do not necessarily contain a free subgroup.
	In these cases, we can appeal to the following result. 
	
	\begin{lemma}\label{uniformexpgrowthquotient}
		Suppose that $\varphi\colon G\to H$ is a surjection. 
		Suppose that $H$ has uniform exponential growth. 
		Then $G$ has uniform exponential growth. 
	\end{lemma}
	\begin{proof}
		Let $S$ be a generating set for $G$. 
		Then $\varphi(S) = \{\varphi(s)\mid s\in S\}$ generates $H$. 
		Since the image of the ball of radius $N$ in $\langle S\rangle$ surjects the ball of radius $N$ in $\langle \varphi(S)\rangle$ and the latter grows exponentially, we see that the former does as well. 
	\end{proof}
	
	We finish this section with several propositions showing various elementary examples where quantitative law alternatives are preserved. 
	We start by showing that these are commensurability invariants. 
	\begin{proposition}\label{prop:commensurabilityinvariant}
		Let $G$ be a group with a finite index subgroup $F$. 
		Then $G\in\lawEither$ if and only if $F\in\lawEither$.
	\end{proposition}
	\begin{proof}
		Since any subgroup of $F$ is a subgroup of $G$ the only if direction is clear. 
		
		Now suppose that $F\in\lawEither$. 
		Let $n$ be the index of $F$ in $G$. 
		We can pass to a further finite index subgroup $K$ which is normal in $G$ whose index is bounded by $n!$. 
		Let $H$ be a subgroup of $G$ generated by a set $S$. 
		Then $H\cap K$ is a subgroup of $K$ and has finite index in $H$. 
		Note that, if $H\cap K$ satisfies the law $w$, then we see that $H$ also satisfies a law of length bounded by $n!\cdot |w|$. 
		
		Now suppose that $F\in \lawFree$ and $H$ does not satisfy a law. 
		Then we see that $H\cap K$ is generated by words in $S$ of length bounded by $2n! + 1$. 
		Since $H$ doesn't satisfy a law, we see by \Cref{virtual_law}, that $H\cap K$ does not satisfy a law and hence we can find an $N$-short free subgroup. 
		Thus we obtain a $(2n!+1)N$-short free subgroup of $H$. 
		Hence $G\in\lawFree$. 
		
		Now suppose $F\in \lawUEG$, we can once again assume that $H$ does not satisfy a law. 
		Thus, $H\cap K$ has uniform exponential growth. 
		We can now appeal to \cite[Corollary 3.6]{ShalenWagreich}, to see that $H$ has uniform exponential growth. 
	\end{proof}
	
	There are some easy cases where quantitative law alternatives are preserved under taking extensions and quotients. 
	We give two elementary such results here which will be used later in the paper. 
	
	\begin{proposition}
		\label{law_alternative:abelian_kernel}
		Suppose that $G$ is a group containing a normal subgroup $N$ which satisifes a law. 
		Then $G$ satisfies the quantitative free subgroup--law alternative if and only if $G/N$ satisfies the quantitative free subgroup--law alternative. 
	\end{proposition}
	\begin{proof}
		Suppose that $G/N$ satisfies the quantitative free subgroup--law alternative.
		Let $H$ be a subgroup of $G$. 
		Consider the image $H/(N \cap H)$ of $H$ in $G/N$.
		If $H/(N \cap H)$ satisfies a law, then we see that $H$ also satisfies a law by \Cref{law-by-law}. 
		If $H/(N \cap H)$ contains a short free subgroup, then we can lift this to a short free subgroup in $H$ by \Cref{shortfreeinquotient}. 
		
		Now for the other direction suppose that $H/(N \cap H)$ is a subgroup of $G/N$ with $H$ a subgroup of $G$.
		If $H$ satisfies a law, then so does any quotient and hence $H/(N \cap H)$ satisfies a law. 
		Suppose that $H$ contains a short rank two free subgroup $F$. 
		Then that image of $F$ in $H/N$ is exactly $F/(F\cap N)$. 
		However, since  $F\cap N$ is both normal in $F$ and satisfies a law we see that it must be trivial. 
		Thus $H/(N \cap H)$ contains a short free subgroup. 
	\end{proof}
	
	\begin{proposition}
		\label{law_alternative:nilpotentkernel}
		Suppose that $G$ is a group containing a normal subgroup $N$ which is  finitely generated and virtually nilpotent. 
		Then $G$ satisfies the quantitative UEG--law alternative if and only if $G/N$ satisfies the quantitative UEG--law alternative. 
	\end{proposition}
	\begin{proof}
		Suppose that $G/N$ satisfies the quantitative UEG--law alternative.
		Let $H$ be a subgroup of $G$. 
		Consider the image $H/(N \cap H)$ of $H$ in $G/N$.
		If $H/(N \cap H)$ satisfies a law, then we see that $H$ also satisfies a law by \Cref{law-by-law}. 
		If $H/(N \cap H)$ contains has uniform exponential growth, then so does $H$ by \Cref{uniformexpgrowthquotient}. 
		
		Now for the other direction suppose that $H/(N \cap H)$ is a subgroup of $G/N$ with $H$ a subgroup of $G$.
		If $H$ satisfies a law, then so does any quotient and hence $H/(N \cap H)$ satisfies a law. 
		
		Now suppose that $H$ is exponentially growing. 
		Now suppose that $S$ is a generating set for $H/(N \cap H)$ and $H/(N \cap H)$ does not satisfy a law. 
		Then, since $N$ is finitely generated virtually nilpotent, we see that all its subgroups satisfy the same property. 
		Hence, $N\cap H$ is finitely generated virtually nilpotent. 
		Let $T$ be a generating set for $N\cap H$. 
		By picking preimages for each $s\in S$ and adding in $T$, we can obtain a generating set $U$ for $H$. 
		The map $\langle U\rangle\to \langle S\rangle$ given by sending each element of $T$ to the identity is distance non-increasing.
		Thus we obtain the inequality $B_n^S\cdot B_n^T\geq B_n^U$. 
		
		Since $\langle T\rangle$ is polynomially growing, we see that $ \lim\limits_{n\to\infty} \frac{\ln(B_n^T)}{n} = 0$ and hence $\lambda_S(H/(N \cap H))\geq \lambda_U(H)$ and thus we are done.  
	\end{proof}

	%
	%

	\section{Extensions of word hyperbolic groups that satisfy the law alternatives}
	\label{section:hyperbolic_extensions}
	
	In this section we will prove \Cref{prop:hyperbolicExtension} and \Cref{thm:twofree} that prove the quantitative law alternatives for extensions of hyperbolic groups.
	We begin with a short lemma regarding automorphisms of hyperbolic groups.
	
	\begin{lemma}\label{lem:ct}
		Let $H$ be a torsion-free hyperbolic group. 
		Let $\phi$ be an automorphism of $H$. 
		Let $a\in H$. 
		If there exist integers $k$ and $l$ such that $\phi(a^k) = a^l$, then $\phi(a) = a^{\pm1}$.
	\end{lemma}
	
	\begin{proof}
		Let $\langle b\rangle$ be the maximal cyclic subgroup of $H$ containing $a$, and
		let $\langle c\rangle$ be the maximal cyclic subgroup containing $\phi(b)$.
		By maximality, we can see that $\langle c\rangle = \langle \phi(b)\rangle$. 
		Thus we can take $c = \phi(b)$. 
		
		Also note that $c$ centralises $a^l$ since $a^l\in \langle c\rangle$. 
		Since the centraliser of $g$ is the maximal cyclic subgroup containing $g$, we see that $c\in \langle b\rangle$.
		By maximality, we obtain $\langle c\rangle = \langle b\rangle$. 
		Since $c = \phi(b)$ we obtain $\phi(b) = b^{\pm 1}$.
		
		We complete the proof by noting that $a = b^n$ for some integer $n$.
		Thus, $\phi(a) = \phi(b^n) = b^{\pm n} = a^{\pm 1}$.
	\end{proof}
	
	We will produce short free subgroups of (hyperbolic 2-free)-by-abelian groups, by exhibiting a single short infinite order element that lives in the kernel.  We consider all conjugates of this infinite order element by each of the generators and use \Cref{lem:ct} to obtain a pair of short words that form a free basis.  
	
	\twofree
	\begin{proof}
		Let $W = \{[t, t']\mid t, t'\in T\}$ and $G = \langle W\rangle$.
		If $G$ is trivial, then $\abrackets{T}$ is abelian and hence not exponentially growing. 
		Thus, we can assume that $G$ is not trivial. 
		Note that the image of each element of $W$ under the map
		$E \to A$ is trivial,
		so $G \le F$.
		Thus, if $G$ is not cyclic, then two elements of $W$ freely generate a free group and we are done. 
		
		Now suppose that $G$ is cyclic. Consider the conjugates $tGt^{-1} < F \cap \abrackets{T}$ for each $t\in T$. 
		If all the conjugates are equal to $G$, then $G$ is a normal subgroup of $\abrackets{T}$. 
		In this case, $\abrackets{T}$ fits into a short exact sequence $0\to G \to \abrackets{T}\to Q\to 0$, where $Q$ is abelian.
		Cyclic-by-abelian groups are virtually nilpotent and so not exponentially growing. 
		
		Now suppose $G$ is not normal in $\langle T\rangle$.
		Let $a$ be a generator for $G$.
		Let $M$ be the maximal cyclic subgroup of $F$ containing $G$. 
		If $tat^{-1}$ commutes with $a$, then $tat^{-1}\in M$ and there are a $k$ and $l$ such that $\phi(a^k) = a^l$.
		Hence by \Cref{lem:ct}, $tat^{-1} = a^{\pm 1}$ and $tGt^{-1} = G$.
		Since $G$ is not normal in $\abrackets{T}$, some conjugate $tGt^{-1}$ does not equal $G$. Hence, $a$ and $tat^{-1}$ generate a free subgroup.  Note that $a$ has $T$-length at most 4, so we have produced a 6-short free subgroup. 
	\end{proof}
	
	Some important examples of 2-free groups come from low-dimensional hyperbolic geometry.  Free groups are 2-free by the Nielsen--Schreier theorem.  It is also straightforward to see that fundamental groups of hyperbolic surfaces are always 2-free.  Moreover, fundamental groups of infinite volume hyperbolic 3-manifolds are also 2-free (see \cite[Theorem~A]{BaumslagShalen} or \cite[Theorem~VI.4.1]{JacoShalen}) as are fundamental groups of 3-manifold groups with sufficiently high first betti number \cite[Corollary~1.9]{ShalenWagreich}.
	Note also that in \Cref{thm:twofree} the kernel group $F$ can be taken to be any group satisfying a very strong version of the Tits alternative as will be shown in \Cref{thm:strongestTits}. 
	The following is an immediate consequence of \Cref{thm:twofree}.
	
	\freebycyclic

	In particular, \Cref{cor:free-by-cyclic} shows that fundamental groups of any mapping torus over a negatively curved surface are in $\lawFree_6$, that is, they satisfy the quantitative 6-short free subgroup--law alternative.  
	We make use of the following result to generalize \Cref{cor:free-by-cyclic} to extensions of hyperbolic groups.    
	The following is well-known to experts and follows immediately from work of Arzhantseva and Lysenok \cite[Theorem~1]{ArzhantsevaLysenok} by considering the constants found in their proof and preceding lemmas. Ideas of the proof can also be found in work of Koubi \cite[Theorem~1.1]{Koubi}.
	
	\begin{theorem}[Uniform Tits alternative for hyperbolic groups]
		\label{thm:hyperbolicTitsAlternative}
		Suppose $H$ is a word hyperbolic group.
		Any finitely generated nonelementary subgroup of $H$ contains a $D$-short free subgroup where $D = D(r, \delta)$, where $r$ is the size of a minimal generating set for $H$ and where $\delta$ is the hyperbolicity constant associated to a fixed generating set of size $r$. 
	\end{theorem}
	
	We are now ready to prove the main result of this section, \Cref{prop:hyperbolicExtension}.  The proof will resemble that of \Cref{thm:twofree} where instead of commutators we will use the law coming from the quotient group.
	
	\hyperbolicExtension
	
	We prove \Cref{prop:hyperbolicExtension} in the case when $L \in \lawFree_N$.  In the setting where $L \in \lawUEG_N$ note that rather than using \Cref{shortfreeinquotient}, one should apply \Cref{uniformexpgrowthquotient} and the rest of the proof follows the same argument.  
	
	\begin{proof}
		Let $n$ be a positive integer and $T\subset E$
		a collection of $n$ elements. 
		Consider $\vphi(\abrackets{T}) < L$. Either $\vphi(\abrackets{T})$ contains an $N$-short free subgroup with respect to $\vphi(T)$ or $\vphi(\abrackets{T})$ satisfies the law $w_L$ because $L \in \lawFree_N$.  
		Applying \Cref{shortfreeinquotient}, we may assume that $\vphi(\abrackets{T})$ does not contain a short free subgroup and hence satisfies the law $w_L$. 
		
		Let $W := \{ w_L(t_1, \dotsc, t_m) \mid  t_i \in T \} \subseteq \abrackets{T}$.  By construction $\abrackets{W} \subset \ker(\vphi) = H$.  
		If $\abrackets{W}$ is nonelementary then $\abrackets{W}$ contains an $D$-short free subgroup where $D = D(r,\delta)$ is the constant from \Cref{thm:hyperbolicTitsAlternative}, so $\abrackets{T}$ would contain a $(D\cdot |w_L|)$-short free subgroup.  
		In light of this, we further assume that $\abrackets{W} < \ker(\vphi) \cap G$ is virtually cyclic. 
		The final step in this proof is the following lemma.

		\begin{lemma}
			\label{hyperbolic_iterated_automorphism}
			Let $H$ be a hyperbolic group and $W\subseteq H$ be a non-trivial finite subset. 
			Let $T \subseteq \Aut(H)$. 
			Let $U_k = \langle \bigcup_{j = -k}^k T^j(W)\rangle$,
			and let $U = \bigcup_{k = 0}^\infty U_k$. 
			Let $B$ be the size of the maximal finite subgroup of $H$.  
			Then either $U_{B+1}$ is non-elementary or $U_{2B} = U$ is virtually cyclic.
		\end{lemma}
		
		\begin{proof}[Proof of lemma]
			We prove this with a sequence of claims. A key point to note is that if $U_\ell = U_{\ell+1}$, then $U = U_\ell$. 
			Let $B$ be the maximal size of a finite subgroup of $H$, note that this is finite by hyperbolicity of $H$. 
			\begin{claim}
				\label{claim:Expanding_finite_subgroup}
				Either $U_B$ is infinite or $U = U_B$ is finite.
			\end{claim} 
			\begin{proof}[Proof of \Cref{claim:Expanding_finite_subgroup}]
				Suppose that $U_{B-1} \neq U_B$.
				Then we have that $U_{\ell - 1} \neq U_\ell$ for all $\ell \leq B$.
				If $U_\ell$ is infinite for some $\ell \leq B$ then so is $U_B$. So suppose $U_B$ is finite. We have an increasing sequence of finite subgroups $U_0\subsetneq \dots\subsetneq U_B$. 
				Hence $|U_B| > B+1$ because $|U_0| \geq 2$ by non-triviality of $W$.
				This gives a contradiction. 
			\end{proof}
			
			We can now suppose that $U_{B}$ is infinite.
			\begin{claim}
				\label{claim:Expanding_virtually_cyclic}
				Suppose $U_\ell$ is infinite and virtually cyclic (henceforth elementary). 
				Then either $U_{\ell+B} = U_{\ell+B+1}$ or $U_{\ell + 1}$ is non-elementary. 
			\end{claim}
			\begin{proof}[Proof of \Cref{claim:Expanding_virtually_cyclic}]
				Throughout let $j\geq \ell$, so $U_j$ is infinite. 
				Elementary groups never have non-elementary subgroups, so suppose that $U_{\ell+B}$ is elementary. 
				Let $C \leq H$ be a maximal virtually cyclic subgroup containing $U_{\ell+B}$. 
				Since $C$ is two-ended, a result of Wall \cite{Wall} shows that $C$ fits into a short exact sequence $1\to F\to C \stackrel{\phi}{\to} Z \to 1$ where $F$ is finite and $Z$ is either $\Z$ or the infinite dihedral group $D_\infty$. 
				In the case that $Z = \Z$, we have that $F = \{c\in C \mid \operatorname{ord}(c)<\infty\}$. 
				By passing to an index 2 subgroup of $C$ we can assume that $Z \cong \Z$. 
				This requires us to pass to index 2 subgroups of each $U_j$. 
				Note that $U_j$ is non-elementary if and only if its index 2 subgroups are. 
				
				Let $\phi\colon C\to Z\cong\Z$ be the map from above. 
				We will show that for each $x\in U_{j}$ the set $\{\phi(t^i(x))\mid t\in T, i\in \Z\}$ has size at most 2. 
				We prove this as follows, for each $t\in T$, we have that $t(C)\cong C$ and hence we have a map $t(C)\to Z$. 
				We will first show that there is a homomorphism $\psi\colon Z\to Z$ such that $\psi(gF) = t(g)F$. 
				Define the map as above, we must check this is well defined. 
				Suppose $gF = g'F$, then $g^{-1}g'\in F$ and hence has finite order.
				Thus we see that $t(g^{-1}g')$ has finite order in $t(C)$ and hence $\phi(t(g^{-1}g')) = 0$ in $Z$ and hence $t(g)F = t(g')F$. 
				We can see that $\psi$ is an isomorphism since $t$ is an isomorphism and we can apply the argument with $t^{-1}$. 
				Since the only isomorphisms $\Z\to \Z$ are $\pm 1$ we obtain that $\{\phi(t^i(x))\mid t\in T, i\in \Z\} = \{\pm \phi(x)\}$. 
				
				Thus we see that $\phi(U_j) = \phi(U_{j+1})$ whenever $U_j$ is infinite. 
				We deduce that if $U_j\neq U_{j+1}$, then $U_j\cap F\neq U_{j+1}\cap F$ but as in the proof of \Cref{claim:Expanding_finite_subgroup} this is a finite subgroup and so can only expand $B$ times.
				Hence taking $j = \ell$ we obtain the desired result. 
			\end{proof}
			This completes the proof of the lemma, since after $B$ steps we will have $U_{B}$ infinite and after another $B$ steps we have a non-elementary subgroup. Thus, we can take $M = 2B$. 
		\end{proof}
		
		Returning to the proof of \Cref{prop:hyperbolicExtension}, we can apply \Cref{hyperbolic_iterated_automorphism} to see that either $U_{2B}$ is non-elementary, in which case $\abrackets{T}$ would contain a $D\cdot(4B + |w_L|)$-short free subgroup,
		or $U$ is elementary. 
		In this latter case we see that $U$ is normal in $\langle T\rangle$. 
		Thus we have a short exact sequence $1\to U \to \langle T \rangle \to \langle \phi(T)\rangle \to 1$.

		We can now pass to a finite index characteristic cyclic subgroup $K$ of $U$ whose index is bounded by a constant $k_1$ depending only on $B$. 
		Since $K$ is characteristic in $U$ and $U$ is normal in $\langle T\rangle$ we see that $K$ is normal in $\langle T\rangle$. 
		Thus we obtain short exact sequences, 
		\begin{center}
			\begin{tikzcd}
				1 \arrow[r] & 
				K \arrow[r] & 
				\abrackets{T} \arrow[r] & 
				Q \arrow[r] & 
				1
			\end{tikzcd}
		\end{center}
		where $Q = \abrackets{T} / K$, and
		\begin{center}
			\begin{tikzcd}
				1 \arrow[r] & 
				F \arrow[r] & 
				Q \arrow[r] & 
				\abrackets{\phi(T)} \arrow[r] & 
				1,
			\end{tikzcd}
		\end{center}
		where $F$ is a finite group of size bounded by $k_1$. 
		Applying \Cref{law-by-law} twice, we see that $\langle T\rangle$ satisfies a law. 
		Since $|F| = k_1$ and $K$ satisfies the commutator law, the law we obtain has length bounded by $4\cdot k_1!\cdot|w_L|$.

		Note that if $H$ admits a $\delta$-hyperbolic Cayley graph with respect to a generating set of size $r$ then $B\leq 2(2r - 1)^{4\delta + 2}$  \cite[{proof of III.$\Gamma$~Theorem~3.2}]{BridsonHaefliger}.
		
		To obtain the final statement we procced as above. 
		If $\langle \phi(T)\rangle$ has exponential growth, then we see that $\langle T\rangle$ has exponential growth and obtain uniformity from \Cref{uniformexpgrowthquotient}. 
		Thus, once again we can assume that $\langle \phi(T)\rangle$ does not have exponential growth and hence satisfies a law. 
		Once again, we have a short exact sequence $1\to U \to \langle T \rangle \to \langle \phi(T)\rangle \to 1$, if $U$ is not virtually cyclic, then we obtain a $D\cdot(4B + |w_L|)$-short free subgroup. 
		Thus, we can assume that $U$ is virtually cyclic. 
		We will show that under these assumptions $\langle T\rangle$ has sub-exponential growth completing the proof. 
		
		Once again, we can pass to a characteristic finite index cyclic subgroup $K$ of $U$. 
		Thus, we obtain a short exact sequence $1\to K \to \langle T \rangle \to Q \to 1$. 
		By passing to an index 2 subgroup of $\langle T\rangle$, we can assume that $K$ is central. 
		Also $Q$ fits into a short exact sequence $1\to F \to Q \to \langle \phi(T)\rangle\to 1,$ where $F$ is finite. 
		Hence $Q$ is quasi-isometric to $\langle \phi(T)\rangle$ and has sub-exponential growth. 
		And $\langle T\rangle$ is quasi-isometric to a central extension of $Q$ and so also has sub-exponential growth by \cite{Zheng}.
	\end{proof}

	%
	%

	\section{Automorphism groups satisfying quantitative alternatives}
	\label{section:automorphism_groups}
	
	In this section we use results from \Cref{section:hyperbolic_extensions} and extend them to prove \Cref{cor:oneendedrelhyp} and \Cref{AutRAAGalternative} that demonstrate that automorphism groups of certain one-ended groups satisfy quantitative subgroup alternatives.  In particular, they combine to give the statement of \Cref{Aut_Combo_Result}.  
	We also show that the automorphism group of any torsion-free one-ended hyperbolic group has locally uniform exponential growth.    
	Our methods apply to automorphism groups by viewing them as extensions by well-studied related groups.

	\begin{lemma}
		\label{Aut from Out: hyperbolic groups}
		Suppose that $H$ is a word-hyperbolic group.  
		Let $O$ be a finite index subgroup of $\Out(H)$. 
		If $O$ satisfies either quantitative law alternative, then so does $\Aut(H)$. 
		Also if $O$ has locally uniform exponential growth so does $\Aut(H)$. 
	\end{lemma}
	\begin{proof}
		By \Cref{prop:commensurabilityinvariant}, we obtain that if $O$ satisfies either quantitative law alternative, then so does $\Out(H)$. 
		Also groups inherit locally uniform exponential growth from their finite index subgroups by intersecting and a result of Shalen and Wagreich \cite[Corollary~3..6]{ShalenWagreich}. 
		Thus, if $O$ has locally uniform exponential growth so does $\Out(H)$. 
		
		Recall that we have short exact sequence 
		\begin{center}
			\begin{tikzcd}
				1 \arrow[r] & 
				\Inn(H) \arrow[r] & 
				\Aut(H) \arrow[r] & 
				\Out(H) \arrow[r] & 
				1
			\end{tikzcd}
		\end{center}
		The inner automorphism group $\Inn(H) \cong H / Z(H)$ where $Z(H)$ denotes the center of $H$.  Since, word-hyperbolic groups are either virtually cyclic or have finite centers $\Inn(H)$ is again hyperbolic.  Thus, we can appeal to \Cref{prop:hyperbolicExtension} to complete the proof. 
	\end{proof}

	\begin{rem}
		It remains open whether either $\Aut(\Fn)$ or $\Out(\Fn)$ have locally uniform exponential growth. 
		From \Cref{Aut from Out: hyperbolic groups} we obtain that if $\Out(\Fn)$ has locally uniform exponential growth then $\Aut(\Fn)$ does as well.  
	\end{rem}
	
	Before proving the quantitative UEG--law alternative for automorphism groups of one-ended hyperbolic groups (\Cref{oneEndedHyperbolic_lawAlternative}) we concentrate on extensions of groups that are not necessarily hyperbolic.  
	
	We say that a group satisfies the \emph{largeness alternative} when any finitely generated subgroup either surjects a nonabelian free group or is free abelian.  Important examples of groups satisfying the largeness alternative are right-angled Artin groups \cite{Baumdisch:2-generatedRaagSubgroups}, residually free groups \cite[Lemma~1]{Baumslag:ResiduallyFree}, and certain finite index subgroups of tubular groups \cite[Theorem~3.7]{Button:TubularStongestTits}.  All such groups are in $\mathcal{L}^{Free}_1$ with law $w = [a,b]$, so satisfy the quantitative free subgroup--law alternative.  
	The largeness alternative has been studied by Antolin and Minasyan who call it the \emph{strongest Tits alternative} and show that it and several other subgroup alternatives are closed under taking graph products \cite[Corollary~6]{AntolinMinasyan}.  An immediate consequence of their result is that graph products of right-angled Artin groups and residually free groups also satisfy the largeness alternative.
	We show the following theorem:
	
	\begin{theorem}
		\label{thm:strongestTits}
		Let $K$ by a group satisfying the largeness alternative with $\vcd(K) \leq d$.  Let $L \in \mathcal{L}^*_N$ be a group satisfying either of the quantitative law alternatives.  Let $E$ be a group fitting into a short exact sequence $1 \to K \to E \stackrel{\vphi}{\to} L \to 1$. Then $E \in \mathcal{L}^*_M$ where $M \geq 2d + |w_L|$ with associated law $[w_L(x_1, \dotsc, x_m), w_L(y_1, \dotsc, y_m)]$.  
	\end{theorem}
	\begin{proof}
		Let $\vphi: E \to L$ be the be the third map in the short exact sequence.  
		Let $T \subset E$ be a finite collection.  
		As in the proof of \Cref{thm:twofree}, we may assume that $\vphi(\abrackets{T})$ satisfies the law $w_L$. Pairs of elements in $K$ either generate nonabelian free groups or a free abelian group because $K$ satisfies the largeness alternative, so we may further assume that $W := \{ w_L(t_1, \dotsc, t_m) \mid t_i \in T \}$ generates a free abelian subgroup of $\ker(\vphi) \cap \abrackets{T}$ with rank at least one.  
		
		As in the proof of \Cref{prop:hyperbolicExtension}, consider $U_{\ell}$ the subgroup of $K$ generated by all conjugates of $W$ by words with $T$-length at most $\ell$. If $U_d$ contains a pair of noncommuting elements then they generate a nonabelian subgroup free subgroup and $G \in \mathcal{L}^*_{2d + |w_L|}$ so we are done.  Otherwise, $U_d$ is a normal free abelian subgroup of $T$.
		
		Normality follows from having considered $(d+1)$ conjugates and the structure of subgroups of finitely generated free abelian groups.  
		Thus, $U_d \normal \abrackets{T}$ with quotient $\abrackets{T}/U_d = \vphi(\abrackets{T})$ satisfying $w_L$.  It follows from \Cref{law-by-law} that $\abrackets{T}$ satisfies the law\\ 
		$[w_L(x_1, \dotsc, x_m), w_L(y_1, \dotsc, y_m)]$ completing the proof.  
	\end{proof}
	
	The subgroup structure afforded by \Cref{thm:strongestTits} lets us characterize extensions of certain non-hyperbolic groups. 
	Recall that the right-angled Artin group associated to a finite simplicial graph has a presentation with one generator for each vertex and whose defining relations are commutators between generators associated with adjacent vertices.  
	
	\begin{coro}
		\label{RAAG_or_RFree_extensions}
		Let $K$ be either a right-angled Artin group or a residually free group.  Let $Q \in \lawEither$ be a group satisfying either quantitative law alternative.  
		Let $E$ be an extension fitting into the exact sequence $1 \to K \to E \to Q \to 1$.  
		Then $E \in \lawEither$.
	\end{coro}
	\begin{proof}
		Right-angled Artin groups and residually free groups satisfy the strongest Tits alternative by work of Antolin and Minasyan \cite{AntolinMinasyan} and Baumslag \cite[Lemma~1]{Baumslag:ResiduallyFree} respectively.
	\end{proof}

	A residually free group is a limit group (also called fully residually free) if and only if it is commutative transitive \cite{Baumslag:ResiduallyFree}.  Commutative transitivity provides the following significant improvement to \Cref{thm:strongestTits}.  
	
	\begin{theorem}
		\label{thm:ctgroups}
		Let $H$ be a limit group. 
		Let $L \in \lawFree_N$ be a group satisfying either quantitative law alternative. 
		Let $G$ be an extension group fitting into a short exact sequence $1 \to H \to G \stackrel{\phi}{\to} L \to 1$. 
		Then $G \in \lawFree_{M}$ where $M = \max\{ N, |w_L| + 2 \}$. 
	\end{theorem}
	
	We first need the following lemma on automorphisms of commutative transitive groups. 
	
	\begin{lemma}
		\label{lem:ctcents}
		Let $H$ be a commutative transitive group. 
		Let $\psi\in \Aut(H)$. 
		If $g\in H$ is such that $\langle \psi^k(g) \mid k\in \Z\rangle$ is non-abelian, then $\langle g, \psi(g)\rangle$ is non-abelian. 
	\end{lemma}
	\begin{proof}
		Suppose that $g$ and $\psi(g)$ commute, i.e. $[g, \psi(g)] = e$. 
		Then $[\psi^k(g), \psi^{k+1}(g)] = e$.
		Since $H$ is commutative transitive, we see that $\langle \psi^k(g)\mid k\in\Z\rangle$ is abelian as all the generators commute.
	\end{proof}
	
	We are now ready to prove \Cref{thm:ctgroups}.
	
	\begin{proof}
		Let $S\subset G$ and let $K = \langle S\rangle$. 
		Let $\phi\colon G\to L$ be the homomorphism from the short exact sequence. 
		Suppose that $\phi(K)$ does not satsify the law $w_L$. 
		Then $\{\phi(s)\mid s\in S\}$ generates an $N$-short free group.
		We can then pull this subgroup back to $G$ to get an $N$-short free subgroup of $G$. 
		
		Now suppose that $\phi(K)$ satisfies $W_L$.
		Let $W = \{w_L(S)\}$. 
		Since $\phi(K)$ satisfies $w_L$ we can see that $W \subseteq H$. 
		
		Since $H$ is a residually free group, every 2-generated subgroup is either abelian or free. 
		Moreover, $H$ is commutative transitive so either $\langle T\rangle$ is abelian or $W$ generates a 1-short free subgroup. 
		If $W$ contains a 1-short free subgroup, then $S$ contains a $|w_L|$-short free subgroup. 
		
		Thus we can assume that $A = \langle W\rangle$ is abelian. 
		Thus $A\leq C(w)$ for any $w\in W$ where $C(w)$ is the centraliser of $w$. 
		Now consider the conjugates $sWs^{-1}$ for $s\in S$. 
		Let $A_s = \langle W, sWs^{-1}\rangle$. 
		
		If each of the $A_s$ are abelian, then by commutative transitivity we see that $A_s\subset C(w)$ for all $s\in S$ and all $w\in W$. 
		Thus, by \Cref{lem:ctcents} we see that conjugation by $s$ preserves $C(w)$ for all $s\in S$.
		Thus the normal closure in $K$ of $A$ is an abelian subgroup $A'$. 
		Since $A$ is the normal closure of the law $w_L$ in $K$ we see that $K/A'$ satisfies $w_L$. 
		Thus $K$ satisfies a law. 
		
		If on the other hand one of the $A_s$ is not abelian, then $T\cup sWs^{-1}$ generates a 1-short free subgroup. 
		Thus, $S$ generates a $|w_L| + 2$-short free subgroup. 
	\end{proof}
	
	Our proofs of the quantitative law alternatives for automorphism groups rely on the following two results.  
	They synthesize well-known results in the literature.  We include the statements and proofs for completeness and ease of reading.
	
	\begin{proposition}
		\label{MCG:quantitative_alternative}
		The mapping class group of a finite-type surface satisfies the quantitative free subgroup--law alternative, that is, $MCG(\Sigma) \in \lawFree_N$ where $N$ and the associated law depends on the complexity of the surface.  
	\end{proposition}
	
	\begin{proof}
		Finitely generated subgroups of the mapping class group of a finite type surface were shown to either be virtually abelian or contain a free subgroup by independent work of Ivanov \cite{Ivanov:Tits_Alternative} and McCarthy \cite[Theorem~A]{McCarthy:Tits_Alternative}.  
		In the first case, the index, $M$ of an abelian subgroup is bounded in terms of the complexity of the surface by work of Birman, Lubotzky, and McCarthy \cite[Theorem~B]{BirmanLubotzkyMcCarthy:Solvable_in_MCG}.  Such groups then satisfy the law $[x,y] \circ z^{M!}$ by \Cref{virtual_law}.  The exponentially growing subgroups were shown by work of Mangahas \cite{Mangahas} to contain $N$--short free subgroups where $N$ is a function of the complexity of the surface.  
		Thus, $\MCG(\Sigma) \in \lawFree_N$.
	\end{proof}
	
	\begin{proposition}
		\label{Linear:quantitative_alternative}
		The general linear group over an algebraically closed field $\mathbb{K}$ satisfies the quantitative free subgroup--law alternative, that is, $\GL(n, \mathbb{K}) \in \lawFree_N$ where $N = N(n)$ and the associated law also only depends on the dimension $n$.
	\end{proposition}
	
	\begin{proof}
		The Tits alternative for linear groups shows that finitely generated subgroups either are virtually solvable or contain a free subgroup.  Both the index $k$ and derived length $c$ of a solvable subgroup can be bounded above by a function of the dimension $n$ (see for example lecture notes of Drutu and Kapovich \cite[Theorem~14.72]{DrutuKapovich:GGTlectures}).  Such groups will satisfy a law of length bounded by $k!\cdot 4^c$ by \Cref{virtual_law}.  Work of Breuillard \cite[Theorem~1.1]{Breuillard} shows that non-virtually solvable subgroups will contain an $N$--short free subgroup where $N = N(n)$.  Thus, $\GL(n, \mathbb{K}) \in \lawFree_N$. 
		
	\end{proof}

	We now turn our attention to the automorphism group of a relatively hyperbolic group.
	Given a group $G$ that is hyperbolic relative to a collection of subgroups $\mathcal{P} = \{P_1, \dots, P_n\}$, we can study the {\em relative outer automorphism group} $\Out(G; \P)$, which is the collection of outer automorphisms preserving $\P$ up to conjugacy. 
	We will focus on the case that $\P$ is a collection of virtually abelian subgroups that are not virtually cyclic. 
	This includes the class of hyperbolic groups for which we can take $\P = \varnothing$, the class of limit groups and more generally groups acting on CAT(0) spaces with isolated flats. 
	In this case, the discussion surrounding Theorem 4.3 in work of Guirardel and Levitt \cite{GuirardelLevitt}, shows that $\Out(G; \P)$ has finite index in $\Out(G)$. 
	They obtain the following: 
	
	\begin{theorem}\cite[Theorem 4.3]{GuirardelLevitt}
		\label{GuirardelLevitt:one-ended_rel-hyp_exact-sequence}
		Let $G$ be hyperbolic relative to $\P = \{P_1, \dots, P_n\}$, with $P_i$ virtually Abelian, not virtually cyclic. 
		Assume that $G$ is one-ended.
		Then $\Out(G)$ has a finite index subgroup $Out^1(G)$ which fits into a short exact sequence
		
		\begin{center}
			\begin{tikzcd}
				1 \arrow[r] & 
				\mathcal{T} \arrow[r] & 
				\Out^1(G) \arrow[r] & 
				\prod\limits_{i=1}^p \MCG^0_T(\Sigma_i)\times \prod\limits_j \Out^0_T(P_j) \arrow[r] & 
				1
			\end{tikzcd}
		\end{center}
		where:
		\begin{itemize}
			\item $\mathcal{T}$ is the group of twists of the canonical elementary JSJ decomposition, thus $\mathcal{T}$ is a finitely generated virtually abelian group. 
			\item $\Sigma_1, \dots, \Sigma_p$ are 2-orbifolds occuring in flexible QH vertices $v$ of the JSJ decomposition. 
			In particular, the groups $\MCG^0_T(\Sigma_i)$ are all commensurable with subgroups of mapping class groups of surfaces. 
			\item $\Out^0_T(P_j)$ is a subgroup of $GL_n(\Z)$ for some $n$, depending on $j$.
		\end{itemize}
	\end{theorem}

	\begin{coro}
		\label{cor:oneendedrelhyp}
		Let $G$ be hyperbolic relative to $\P = \{P_1, \dots, P_n\}$, with $P_i$ virtually abelian, not virtually cyclic. 
		Assume that $G$ is one-ended.
		Then $\Out(G)$ satisfies the quantitative free subgroup--law alternative.
	\end{coro}
	\begin{proof}
		By \Cref{law_alternative:abelian_kernel}, it suffices to note that the quantitative free subgroup--law alternative is satisfied by both $\MCG(\Sigma_i)$ and $\Out(P_j)$.
		Each $\MCG(\Sigma_i)$ satisfies a quantitative law alternative because it is commensurable with a subgroup of the mapping class group of a finite type surface and \Cref{MCG:quantitative_alternative}.  
		Each $\Out(P_j)$ is a subgroup of $\GL(n, \Z)$, so satisfies a quantitative law alternative by \Cref{Linear:quantitative_alternative}. 
	\end{proof}
	
	\begin{theorem}
		\label{relhyp_Aut:alternative}
		Let $G$ be a torsion-free group that is hyperbolic relative to a family of 
		free
		abelian subgroups $\P = \{ P_1, \dotsc, P_n \}$.  If $G$ is one-ended, then $\Aut(G)$ satisfies the quantitative UEG--law alternative.
	\end{theorem}
	
	Before proceeding we will need the following variation of a result \cite[Lemma~5.3]{Xie}. 
	This stronger result is not explicitly stated in \cite{Xie}, but follows immediately from Xie's proof. We include a proof summary for completeness, to highlight that cocompactness is only needed to exhibit a particular isometry of the cusp space, and to establish notation and a setting that will be used to prove \Cref{relhyp_Aut:alternative}.
	
	\begin{proposition}
		\label{relhyp:alternative_LUEG}
		Let $G$ be a torsion-free group that is hyperbolic relative to a family of subgroups each of which satisfies a quantitative free subgroup--law alternative (has locally uniform exponential growth).  
		Then $G$ also satisfies a quantitative free subgroup--law alternative (has locally uniform exponential growth).
	\end{proposition}
	
	\begin{proof}[Proof summary]
		Let $S$ be a fixed generating set for $G$. 
		Let $W \subset G$ be an arbitrary finite collection of elements.  
		We assume that $W$ generates a subgroup $H$ that is not conjugate into the peripherals $\P$ or else we are done by hypothesis.  
		
		There exists a proper hyperbolic geodesic metric space $X$ and a $G$-invariant subset $Y \subset X$ such that the complement $X \smallsetminus Y =: \ccB$ is a collection of $200 \delta$--separated horoballs (see \cite[Section~3.2]{Hruska:relhyp_relQC} for a discussion on cusp uniform actions and horoballs as well as \cite[Section~5]{Bowditch:relhyp}).
		By work of Bowditch  \cite[Proposition~6.13]{Bowditch:relhyp}, 
		$G$ acts properly and cocompactly on $Y$.
		For concreteness, we can 
		and will 
		use a construction of Groves and Manning \cite[Definition~3.12]{GrovesManning} to
		assume that such a space $X$ is obtained from the Cayley graph $\Cay(G, S)$ by equivariantly attaching combinatorial horoballs to conjugates of the peripheral subgroups in $\P$.  
		The corresponding subspace $Y$ can be obtained by equivariantly truncating the combinatorial horoballs at an appropriate height.
		
		It will suffice to exhibit a uniformly (independent of choice of $W$) short word in $H$ that moves every point of $X$ at least $100\delta$.  
		Let $\Omega \subset Y$ be a compact fundamental domain for the action of $G$ on $Y$.  Let $a := diam(\Omega)$.  By properness of $Y$ and the action of $G$ on $Y$
		\[
		m := \left | \{ g \in G \mid g\Omega \cap \text{N}_{100\delta + 2a}(\Omega) \neq \varnothing \} \right |
		\]
		is finite.
		Fix a basepoint $p \in \Omega$.
		By choice of $m$ there exists a word $h_0 \in H$ with $W$-length at most $m$ that moves $p$ more than $100\delta+2a$.  
		The word $h_0$ necessarily moves every point of $X$ the desired amount.  
		Indeed, suppose there exists a point $x \in X$ such that every element of $W$ moves $x$ less than $100\delta$.  
		By $200\delta$--separation, $x$ cannot be contained in a horoball or else $H$ would fix the horoball and be conjugate into a peripheral subgroup.  
		Hence, $x \in Y$.  There is an element of $g \in G$ such that $g(x) \in \Omega$.  
		By the triangle inequality and the choice of $h_0$, we have that $d_X(g(x), h_0g(x)) \geq 100\delta$.
		
		By work of Xie \cite[Proposition~4.1]{Xie} building upon that of Koubi \cite[Proposition~3.2]{Koubi}, the existence of such an element guarantees existence of a loxodromic $h_1 \in \Isom(X)$ with $W$-length at most $m+1$.
		\cite[Lemma~3.1]{Xie} gives a uniform constant depending $k$ only on hyperbolicity of $X$ and the separation of horoballs such that the power $h_2 = h_1^k$ has displaces every point of $Y$ by at least $200\delta$.
		
		If some element of $W$ conjugates $h_2$ to an independent loxodromic $wh_2w^{-1}$, and \cite[Section~4]{Xie} demonstrates that $\abrackets{w, wh_2w^{-1}}$ is a free group, whence $H$ has uniform exponential growth. 
		Otherwise, $W$ stabilizes the endpoints of an axis of $h_2$ in the Bowditch boundary then $H$ is virtually cyclic hence linearly growing.
		Since by assumption $G$ is torsion free, virtually cyclic subgroups are infinite cyclic.  Thus, $H$ satisfies the commutator law.
	\end{proof}

	\begin{proof}[Proof of \Cref{relhyp_Aut:alternative}]
		Let $d$ be the maximum rank of the free abelian subgroups $P_i$ . 
		
		Let $T \subset \Aut(G)$ be an arbitrary finite collection of automorphisms.
		As in \Cref{prop:hyperbolicExtension}, we may assume that the image in $\Out(G)$ satisfies the law $w_{Out}$ provided by \Cref{cor:oneendedrelhyp}.
		Take 
		\[
		W := \{ w_{Out}(t_1, \dotsc, t_m) \mid t_i \in T \} \subset \abrackets{T} \cap \ker(\Aut(G) \to \Out(G)).
		\]
		Let $W_k := \bigcup_{j = -k}^k T^j(W)$.  Note that $W_k < \Inn(G)$ and because $G$ is relatively hyperbolic and torsion free, $\Inn(G) \cong G$.  The collection $W_k$ each generate an ascending chain of subgroups $\{U_k := \abrackets{W_k}\}_{k = 0}^\infty$ such that the union $U = \bigcup_{k = 0}^\infty U_k$ is the normal closure $\llangle W \rrangle$ in $\abrackets{T}$.
		As before, we are done if either there exists a uniform bound on $k$ such that $U_k$ is non-elementary or $\llangle W \rrangle$ satisfies a law.
		
		Assume that $U_0$ is elementary.
		Under this assumption, $U_0$ is either contained in some conjugate $P^g$ of one of the peripheral subgroups of $G$, or $U_0$ acts loxodromically on the cusp space $X$ described in \Cref{relhyp:alternative_LUEG}.  Because $G$ is assumed to be torsion free $U_0 \cong \Z$.
		
		Suppose for now that $U_0 \leq P^g$.
		If $P^g$ is fixed by all the elements of $T$ then any subgroup of $P^g$ remains inside of this peripheral subgroup and thus $U$ is free abelian and satisfies a law.  
		Thus, $\abrackets{T}$ satisfies a law by \Cref{law-by-law}.
		If on the other hand there exists $t \in T$ such that $t(P^g) \neq P$ then  $U_1$ is not contained in a conjugate of any peripheral subgroup. 
		Indeed, peripheral subgroups necessarily have finite intersection, but $U_0 \leq P^g$ is infinite since $G$ was assumed to be torsion free.

		It remains only to consider when $U_1$ elementary and not contained in a peripheral.  In particular, $U_1 \cong \Z$ acts loxodromically on $X$.  If the elements of $T$ all fix the endpoints $\bdry_\infty U_1 \subset \bdry_\infty X$, where $\bdry_\infty$ denotes the visual boundary, then $U = U_1 \normal \abrackets{T}$ and $\abrackets{T}$ satisfies a law by \Cref{law-by-law}. 
		Finally, if $T$ does not stabilize $\bdry_\infty U_1$, then $U_2$ is non-elementary, so contains an $N$--short free subgroup by \Cref{relhyp:alternative_LUEG} where $N = N(G)$.  
		In particular, this short free subgroup witnesses that $\abrackets{T}$ has uniform exponential growth.
		
	\end{proof}

	When $G$ is hyperbolic and torsion free the situation improves significantly because it can be shown that virtually solvable subgroups of $\Aut(G)$ are virtually nilpotent. 
	
	\begin{theorem}
		\label{oneEndedHyperbolic_lawAlternative}
		Let $G$ be a torsion-free one-ended hyperbolic group. 
		Then $\Aut(G)$ has locally uniform exponential growth.
	\end{theorem}
	\begin{proof}
		The group $G$ is hyperbolic relative $\P = \varnothing$.
		Hence, by \Cref{GuirardelLevitt:one-ended_rel-hyp_exact-sequence} there is a finite index subgroup $Out^1(G)$ of $\Out(G)$ which fits into a short exact sequence: 
		\begin{center}
			\begin{tikzcd}
				1 \arrow[r] & 
				\mathcal{T} \arrow[r] & 
				\Out^1(G) \arrow[r] & 
				\prod\limits_{i = 1}^p \MCG^0_T(\Sigma_i) \arrow[r] & 
				1
			\end{tikzcd}
		\end{center}
		Moreover, by \cite{Levitt}, this sequence is a central extension. 
		
		Now let $H$ be a finitely generated subgroup of $\Out^1(G)$. 
		Let $\bar{H}$ be its image in the quotient $\prod_{i=1}^p \MCG^0_T(\Sigma_i)$. 
		As in the proof of \Cref{cor:oneendedrelhyp}, we have that $\bar{H} \in \lawFree_N$ where $N$ is bounded in terms of the orbifolds $\Sigma_i$ by \Cref{MCG:quantitative_alternative}. 
		If it contains a short free subgroup we can lift this to $H$ by \Cref{shortfreeinquotient} and we are done. 
		Thus we can assume $\bar{H}$ is virtually abelian. 
		
		This reduces us to the case that $H$ fits into a short exact sequence
		\begin{center}
			\begin{tikzcd}
				1 \arrow[r] & 
				T' \arrow[r] & 
				H \arrow[r] & 
				A \arrow[r] & 
				1
			\end{tikzcd}
		\end{center}
		where $T'$ is the free abelian group $H\cap \mathcal{T}$ and $A$ is a virtually abelian subgroup of the quotient $\prod_{i=1}^p \MCG^0_T(\Sigma_i)$. 
		Thus, we see that $H$ is a central extension of a virtually abelian group and hence is virtually nilpotent. 
		
		Thus, we can conclude that $\Out^1(G)$ has locally uniform exponential growth. 
		Also we know from \Cref{cor:oneendedrelhyp} that $\Out(G)$ satisfies the quantitative free subgroup--law alternative. 
		Thus we can appeal to \Cref{Aut from Out: hyperbolic groups} to complete the proof. 
	\end{proof}

	Thus far, we have only used \Cref{thm:strongestTits} in the capacity of showing that extensions of free abelian groups satisfy a quantitative law alternative.  The statement is, however, significantly more general than this.  Indeed we will demonstrate how \Cref{RAAG_or_RFree_extensions} is enough for us to show that autmorphism groups of certain right-angled artin groups satisfy a quantitative law alternative.

	\begin{theorem}
		\label{AutRAAGalternative}
		Let $A_\Gamma$ be the right-angled Artin group associated to a graph, $\Gamma$ without a separating intersection of links (SIL).  
		The automorphsim group $\Aut(A_\Gamma) \in \lawUEG$, that is, it satisfies the quantitative UEG--law alternative. 
	\end{theorem}
	
	\begin{proof}
		We proceed as in the proof for automorphisms of one-ended hyperbolic groups. \Cref{oneEndedHyperbolic_lawAlternative}
		
		The inner automorphism group of a RAAG is again a RAAG. To see this note that a RAAG has non-trivial centre if and only if the defining graph decomposes as a join of a complete graph (corresponding to the centre) and another graph $\mathcal{G}$. 
		The inner automorphism group is then the RAAG corrsponding to $\mathcal{G}$. 
		From the short exact sequence relating the outer, inner, and full automorphism groups and \Cref{RAAG_or_RFree_extensions} we see that is sufficies to show that $\Out(A_\Gamma)$ satisfies the desired alternative.  
		
		By a result of Guirardel and Sale \cite[Theorem~2(2)]{GuirardelSale:AutRAAG}, there exists a finite index subgroup $\ccO \leq \Out(A_\Gamma)$ fitting into an exact sequence of the form
		
		\begin{center}
			\begin{tikzcd}
				1 \arrow[r] & N \arrow[r] & \ccO \arrow[r, "\vphi"] & \prod\limits_{i = 1}^k SL(n_i, \Z) \arrow[r] & 1
			\end{tikzcd}
		\end{center}
		
		where $N$ is virtually nilpotent. 
		Let $G$ be any finitely generated subgroup of $\ccO$. 
		The product $\prod\limits_{i = 1}^k SL(n_i, \Z) \in \lawFree_N$ where $N$ depends on the $n_i$ by \Cref{Linear:quantitative_alternative}.
		If the quotient $\vphi(G)$ contains a short free subgroups, then we can lift it to $G$ by \Cref{shortfreeinquotient}.
		Otherwise, $\vphi(G)$ satisfies a law.  
		Hence, $G$ is a nilpotent-by-law group and therefore satisfies a law by \Cref{law-by-law}.  
	\end{proof}

	\section{Acknowledgements}
	
	The first author was funded by the Deutsche Forschungsgemeinschaft (DFG, German Research Foundation) under Germany's Excellence Strategy EXC 2044--390685587, Mathematics M\"unster: Dynamics--Geometry--Structure.
	The second author acknowledges support from U.S. National Science Foundation grants DMS
	1107452, 1107263, 1107367 "RNMS: Geometric Structures and Representation Varieties" (the
	GEAR Network).
	The third author acknowledges funding from ISF grant \# 660/20, the Temple University Graduate School, and U.S. National Science Foundation DMS–1907708.

	\bibliographystyle{alpha}
	\bibliography{ResearchBib}

\end{document}